\documentclass[10pt, reqno]{amsart}
\usepackage{graphicx, amssymb, amsmath, amsthm}
\numberwithin{equation}{section}

\def\pa{\partial}

\let\Re=\undefined\DeclareMathOperator*{\Re}{Re}
\let\Im=\undefined\DeclareMathOperator*{\Im}{Im}

\newcommand{\R}{\mathbb{R}}
\newcommand{\C}{\mathbb{C}}
\newcommand{\nsc}{\vert\nabla\vert^{s_c}}
\newcommand{\eps}{\varepsilon}

\newcommand{\norm}[1]{\| #1\|}
\newcommand{\xonorm}[3]{\|#1\|_{L_t^{#2}L_x^{#3}}}
\newcommand{\xonorms}[2]{\|#1\|_{L_{t,x}^{#2}}}

\newtheorem{theorem}{Theorem}[section]

\newtheorem{lemma}[theorem]{Lemma}

\newtheorem{corollary}[theorem]{Corollary}

\newtheorem{proposition}[theorem]{Proposition}

\theoremstyle{definition}
\newtheorem{definition}[theorem]{Definition}
\newtheorem{remark}[theorem]{Remark}

\makeatletter
\newcommand{\Extend}[5]{\ext@arrow0099{\arrowfill@#1#2#3}{#4}{#5}}
\makeatother

\begin{document}
\title[Radial defocusing NLS in dimension four]{The radial defocusing energy-supercritical NLS in dimension four}

\author[C. Lu]{Chao Lu}
\address{The Graduate School of China Academy of Engineering Physics, P. O. Box 2101, Beijing, China, 100088}
\email{luchao@mail.bun.edu.cn}

\author[J. Zheng]{Jiqiang Zheng}
\address{Universit\'e Nice Sophia-Antipolis, 06108 Nice Cedex 02, France}
\email{zhengjiqiang@gmail.com, zheng@unice.fr}

\begin{abstract}
We consider the radial defocusing nonlinear Schr\"odinger equations
$iu_t+\Delta u=|u|^{p}u$ with supercritical exponent $p>4$ in four space dimensions,  and prove that any radial solution that remains bounded in the critical
Sobolev space must be global and scatter.

\end{abstract}

 \maketitle

\begin{center}
 \begin{minipage}{100mm}
 \end{minipage}
 \end{center}

\section{Introduction}

\noindent This paper is
devoted to  study the initial-value problem for defocusing
nonlinear Schr\"odinger equations of the form
\begin{align} \label{equ1.1}
\begin{cases}    (i\partial_t+\Delta)u= |u|^pu,\quad
(t,x)\in\R\times\R^4,
\\
u(0,x)=u_0(x)\in H^{s_c}(\R^4),
\end{cases}
\end{align}
where $u(t,x):\R\times\R^4\to \C$, and $s_c=2-\tfrac2p$ with $p>4$.  We prove that for $s_c>3/2$ (i.e. $p>4$), any radial maximal-lifespan solution that remains uniformly bounded in $\dot{H}_x^{s_c}(\R^4)$ must be global and scatter. In \cite{DMMZ,MMZ}, the authors proved the analogous statement for \eqref{equ1.1} with $2<p\leq4$ and non-radial initial data. In this paper, we treat the remaining case $p>4$ but for radial initial data.

The equation \eqref{equ1.1} has the scaling invariance symmetry:
\begin{equation}\label{eq:scaling}
u_\lambda(x,t)=\lambda^{\frac{2}{p}}u(\lambda
x,\lambda^2t),
\end{equation}
in the sense that both the equation and the $\dot{H}^{s_c}(\R^4)$-norm are
invariant under the scaling transformation:
$$\|u_\lambda(t,\cdot)\|_{\dot{H}^{s_c}_x(\R^4)}=\|u(\lambda^2t,\cdot)\|_{\dot{H}^{s_c}_x(\R^4)}.$$
This scaling defines a notion of \emph{criticality} for \eqref{equ1.1}.   If we take $u_0\in\dot{H}_x^s(\R^4)$, then for $s=s_c$ we call the problem \eqref{equ1.1} \emph{critical}. For $s>s_c$ we call the problem \emph{subcritical}, while for $s<s_c$ we call the problem \emph{supercritical}.

If the solution $u$ of \eqref{equ1.1} has sufficient decay at
infinity and smoothness, it conserves  mass and energy:
\begin{equation}
\begin{split}
M(u)&=\int_{\mathbb{R}^4}|u(t,x)|^2dx=M(u_0),\\
E(u)&=\frac12\int_{\mathbb{R}^4}|\nabla
u|^2dx+\frac1{p+2}\int_{\mathbb{R}^4}|u(t,x)|^{p+2}dx=E(u_0).
\end{split}
\end{equation}
As similarly explained in \cite{Cav},  the above quantities are also
conserved for the energy solutions $u \in C^0_t(\mathbb{R},
H^1(\mathbb{R}^4))$. We call $\dot{H}^1_x(\R^4)$ the energy space.

There is a lot of works on the  global well-posedness and scattering
for Schr\"odinger equations
$$i\pa_tu-\Delta u=\pm|u|^pu,\quad (t,x)\in\R\times\R^d\quad\text{(NLS)}$$ especially for mass-critical $(p=4/d)$ or
energy-critical $(p=4/(d-2),~d\geq3)$ (NLS), most notably by Bourgain \cite{Bou},
Colliander, Keel, Staffilanni, Takaoka and Tao \cite{CKSTT07}, Kenig
and Merle \cite{KM} and Killip and Visan \cite{KV20101,KV} and Visan
\cite{Visan2007,Visan2011}, and Dodson\cite{Dodson4} for the energy-critical case and Tao,
Visan and Zhang \cite{TVZ2007}, Killip, Tao and Visan
\cite{KTV2009}, Killip, Visan and Zhang \cite{KVZ2008} and Dodson
\cite{Dodson3,Dodson2,Dodson1,Dodson} for the mass-critical case.

So far, there is no technology for treating large-data (NLS) without
some a priori control of a critical norm other than the
energy-critical (NLS) and mass-critical (NLS) at the present moment.
Kenig-Merle \cite{KM2010} first showed that if the radial solution $u$ to NLS is
such that $u\in L_t^\infty(I; \dot{H}^{s_c}(\R^3))$ with
$s_c=1/2$, then $u$ is
global and scatters. They were able to handle this case by making
use of their concentration compactness technique (as in \cite{KM}),
together with the Lin-Strauss Morawetz inequality which scales like
$\dot{H}^\frac12_x(\R^3)$. Lately,  such result has been extended to high dimensional and inter-critical cases by Murphy \cite{Mu,Mur2,Mur}.
In \cite{KV2010}, Killip--Visan
proved some cases of  the energy-supercritical regime. In
particular, they deal with the case of a cubic nonlinearity for
$d\geq5$, along with some other cases for which $s_c>1$ and
$d\geq5$. The restriction to high dimensions stems from
 the so-called double Duhamel trick;
see \cite{KV2010}  for more details. Recently, by making use of the
tool ``long time Strichartz estimate" \`a la Dodson \cite{Dodson3}, and  a frequency-localized interaction Morawetz inequality, the authors
\cite{DMMZ,MMZ} treat the case $2<p\leq 4$ (i.e. $1<s_c\leq3/2$) in four space dimensions. In this paper, we address
the remaining cases $p>4$ but for radial initial data, where the techniques in \cite{DMMZ,MMZ} break down.

Now, let us make the notion of a solution more precise.
 A function $u:I\times\R^4\to\mathbb{C}$ on a non-empty
time interval $I\ni 0$ is a \emph{solution} to
\eqref{equ1.1} if
it belongs to $C_t\dot{H}_x^{s_c}(J\times\R^4)\cap L_{t,x}^{3p}(J\times\R^4)$
for any compact interval $J\subset I$ and
obeys the Duhamel formula
\begin{equation}\label{duhamel.1}
u(t)=e^{it\Delta}u_0-i\int_{0}^te^{i(t-s)\Delta}F(u(s))\,ds
\end{equation}
for each $t\in I$. We call $I$
the \emph{lifespan} of $u$. We say that $u$ is a \emph{maximal-lifespan solution}
if it cannot be extended to any strictly larger interval.
We call $u$ \emph{global} if $I=\R.$

For a solution $u:I\times\R^4\to\C$ to \eqref{equ1.1}, we define the \emph{scattering size} of $u$ on $I$ by
            \begin{equation}\label{scattersize1.1}
    S_I(u):=\int_I\int_{\R^4}\vert u(t,x)\vert^{3p}\,dx\,dt.
                \end{equation}
If there exists $t_0\in I$ such that $S_{[t_0,\sup I)}(u)=\infty$, we say that $u$ \emph{blows up forward in time}. If there exists $t_0\in I$ such that $S_{(\inf I,t_0]}(u)=\infty$, we say that $u$ \emph{blows up backward in time}.

 If $u$ is a global and obeys $S_{\R}(u)<\infty$, then  standard arguments show that $u$ \emph{scatters} in the sense
tha there exist unique $u_{\pm}\in\dot{H}^{s_c}_x(\R^4)$ such that
        $$\lim_{t\to\pm\infty}\|u(t)-e^{it\Delta}u_{\pm}\|_{\dot{H}_x^{s_c}(\R^4)}=0.$$

Our main result is the following

\begin{theorem}\label{theorem}
Let $s_c>3/2$, i.e. $p>4$. Suppose
$u:~I\times\R^4\to\C$ is a radial maximal-lifespan solution to
\eqref{equ1.1} such that
\begin{equation}\label{assume1.1}
u\in L_t^\infty\dot{H}_x^{s_c}(I\times\R^4).
\end{equation} Then $u$ is global and scatters, with
    $$S_\R(u)\leq C(\|u\|_{L_t^\infty(\R,\dot{H}_x^{s_c}(\R^4))})$$
for some function $C:[0,\infty)\to[0,\infty).$
\end{theorem}

Let us turn now to an outline of the arguments we will use to establish Theorem~\ref{theorem}.

\subsection{Outline of the proof of Theorem \ref{theorem}} We will argue by contradiction.
 For any $0\leq E_0<+\infty,$ we  define
$$L(E_0):=\sup\Big\{S_I(u):~u:~I\times\R^4\to \C\ \text{such\ that}\
\sup_{t\in I}\big\|u\big\|_{\dot H^{s_c}_x(\R^4)}^2\leq E_0\Big\},$$
where the supremum ranges all solutions $u:~I\times\R^4\to
\C$ to \eqref{equ1.1} satisfying $\big\|u\big\|_{\dot
H^{s_c}(\R^4)}^2\leq E_0.$ Thus, $L:\ [0,+\infty)\to [0,+\infty)$ is
a non-decreasing function. Moreover, from  the small data theory (via Strichartz estimates and contraction mapping, cf. \cite{Cav, KVnote}), one has
$$L(E_0)\lesssim E_0^\frac12\quad \text{for}\quad E_0\leq\eta_0^2,$$
where $\eta_0=\eta(d)$ is the threshold from the small data theory.

It follows from the stability theory  that
$L$ is continuous. Thus, there must exist a unique critical
$E_c\in(0,+\infty]$ such that $L(E_0)<+\infty$ for $E_0<E_c$ and
$L(E_0)=+\infty$ for $E_0\geq E_c$. In particular, if
$u:~I\times\R^4\to \C$ is a maximal-lifespan solution to
\eqref{equ1.1} such that $\sup\limits_{t\in I}\big\|u\big\|_{\dot
H^{s_c}_x(\R^4)}^2< E_c,$ then $u$ is global and moreover,
$$S_\R(u)\leq L\big(\big\|u\big\|_{L_t^\infty(\R;\dot H^{s_c}(\R^4))}\big).$$ The proof of Theorem \ref{theorem} is
equivalent to show $E_c=+\infty.$ We
argue by contradiction. The failure of Theorem \ref{theorem} would
imply the existence of very special class of solutions. On the other
hand, these solutions have so many good properties that they do not
exist. Thus we get a contradiction. While we will make some further
reductions later, the main property of the special counterexamples
is almost periodicity modulo symmetries:

\begin{definition}[Almost periodic solutions]\label{AP}
Let $s_c>0.$ A solution $u:I\times\R^4\to\C$ to \eqref{equ1.1} is called \emph{almost periodic }(\emph{modulo symmetries}) if $u\in L_t^\infty\dot{H}_x^{s_c}(I\times\R^4)$ and there exist functions
$N:~I\to\R^+,$  $x:I\to\R^4$  and $C:\R^+\to\R^+$ such that
for all $t\in I$ and $\eta>0$,
\begin{equation}\label{apss}
\int_{|x-x(t)|\geq\frac{C(\eta)}{N(t)}}\big||\nabla|^{s_c}
u(t,x)\big|^2\,dx+\int_{|\xi|\geq C(\eta)N(t)}|\xi|^{2s_c}\vert\hat
u(t,\xi)\vert^2\,d\xi\leq\eta.
\end{equation}
We call $N(t)$ the \emph{frequency scale function},  $x(t)$ the \emph{spatial center function}, and
$C(\eta)$ the \emph{compactness modulus function}.
\end{definition}

\begin{remark}\label{remark:radial} $(i)$ As a consequence of radiality, the solutions we consider can only concentrate near the spatial origin. In particular, we may take $x(t)\equiv 0$.

$(ii)$
The Arzel\`a--Ascoli theorem tells us that a family of functions $\mathcal{F}$ is precompact in $\dot{H}_x^{s_c}(\R^4)$
if and only if it is norm-bounded and there exists a compactness modulus function $C(\eta)$ such that
    $$\int_{\vert x\vert\geq C(\eta)}\big\vert\nsc f(x)\big\vert^2\,dx+\int_{\vert\xi\vert\geq C(\eta)}\vert\xi\vert^{2s_c}\vert \hat{f}(\xi)\vert^2\,d\xi\leq\eta$$
uniformly for $f\in\mathcal{F}$. Thus we see that a solution $u:I\times\R^4\to\C$ is radial almost periodic if and only if
    $$\{u(t):t\in I\}\subset\{\lambda^{2/p}f(\lambda x):\lambda\in(0,\infty),\; \text{and}\; f\in K\}$$
for some compact $K\subset \dot{H}_x^{s_c}(\R^4).$ We deduce that there exists a function $c:\R^+\to\R^+$ such that
\begin{equation}\label{xiaoc}
\int_{|x|\leq\frac{c(\eta)}{N(t)}}\big||\nabla|^{s_c}
u(t,x)\big|^2\,dx+\int_{|\xi|\leq c(\eta)N(t)}|\xi|^{2s_c} |\hat{u}(t,\xi)|^2\,d\xi\leq\eta
\end{equation}
for all $t\in I$.

\end{remark}

By the same argument as in \cite{DMMZ,KVnote,MMZ}, we can show that the failure of Theorem \ref{theorem} would
imply the existence of the almost periodic solutions as follows.

\begin{theorem}[Reduction to almost periodic solutions]\label{thm:AP1} If Theorem~\ref{theorem} fails, then there exists a radial maximal-lifespan solution $u:I\times\R^4\to\C$ to \eqref{equ1.1} that is almost periodic and blows up in both time directions.
\end{theorem}

Furthermore, one can adopt the proof of
Lemma 5.18 in \cite{KVnote} to prove that the almost periodic
solutions satisfy the following local constancy property:
\begin{lemma}[Local constancy]\label{local constancy} Let $u:I\times\R^4\to\C$ be a maximal-lifespan almost periodic solution to \eqref{equ1.1}. Then there exists $\delta=\delta(u)>0$ such that for all $t_0\in I$,
    $$[t_0-\delta N(t_0)^{-2},t_0+\delta N(t_0)^{-2}]\subset I.$$
Moreover,  $N(t)\sim_u N(t_0)$
for $\vert t-t_0\vert\leq\delta N(t_0)^{-2}.$
\end{lemma}

\begin{definition}[Interval of local constancy]\label{def:localconstancy} Let $u:I\times\R^4\to\C$ be a
almost periodic solution. We can divide $I$ into consecutive
intervals $J_k$ such that
$$\|u(t,x)\|_{L_{t,x}^{3p}(J_k\times\R^4)}=1,~\text{and}~N(t)\equiv N_k,~t\in J_k.$$
These intervals are called as the intervals of local constancy. If
$J\subset I$ is a union of consecutive intervals of local
constancy, then we have
\begin{equation}
\sum_{J_k}N(J_k)^{1-2s_c}\sim\int_JN(t)^{3-2s_c}dt.
\end{equation}
\end{definition}

Lemma~\ref{local constancy} provides information about the behavior of the frequency scale at blowup (cf. \cite[Corollary~5.19]{KVnote}):
\begin{corollary}[$N(t)$ at blowup]\label{constancy cor}
 Let $u:I\times\R^4\to\C$ be a maximal-lifespan almost periodic solution to \eqref{equ1.1}. If  $T$ is a finite endpoint of $I$ then $N(t)\gtrsim_u | T-t|^{-\frac{1}{2}}.$
\end{corollary}

We also need the following result, which relates the frequency scale
function of an almost periodic solution to its Strichartz norms.
\begin{lemma}[Spacetime bounds]\label{fspssn} Let $u:I\times\R^4\to\C$ be an almost periodic solution to \eqref{equ1.1}. Then, there holds
$$\int_I N(t)^2\,dt\lesssim_u\big\||\nabla|^{s_c} u\big\|_{L_t^{q}L_x^{r}(I\times\R^4)}^q\lesssim_u 1+\int_I N(t)^2\,dt,$$
for $(q,r)$ admissible $($see Definition \ref{def1} below$)$ with $q<\infty$.
\end{lemma}
\begin{proof}
See Lemma 5.21 in \cite{KVnote}.
\end{proof}

We now refine the class of almost periodic solutions that we consider. By rescaling arguments as in \cite{KTV2009, KV20101, TVZ2007}, we can guarantee that the almost periodic solutions we consider do not escape to arbitrarily low frequencies on at least half of their maximal lifespan, say $[0,T_{\rm max})$. Using Lemma~\ref{local constancy} to divide $[0,T_{\rm max})$ into characteristic subintervals $J_k$, we arrive at the following theorem.

\begin{theorem}[Two scenarios for blowup]\label{threenew} If Theorem \ref{theorem} fails, then there exists a raidal almost periodic solution $u:[0,T_{\rm max})\times\R^4\to\C$ that blows up forward in time and satisfies
\begin{equation}\label{equ:assumpt}
u\in L_t^\infty \dot H_x^{\frac32}([0,T_{max})\times\R^4).
\end{equation}
Furthermore, we may write $[0,T_{max})=\cup_k J_k$, where
\begin{equation}\label{E:Ng1}
N(t)\equiv N_k\geq 1\quad\text{for}\quad t\in J_k,\quad\text{with}\quad |J_k|\sim_u N_k^{-2}.
\end{equation}
             We classify $u$ according to the following two scenarios: either
             \begin{align}\label{equ:rapid}
& \int_0^{T_{\rm max}} N(t)^{3-2s_c}\,dt<\infty \quad (\text{rapid frequency-cascade solution}),
\end{align}
or
\begin{align}
\label{equ:quasi}
&\int_0^{T_{\rm max}} N(t)^{3-2s_c}\,dt=\infty \quad (\text{quasi-soliton solution}).
\end{align}
\end{theorem}

In view of this theorem, our goal is to preclude  the possibilities
of all the scenarios in the sense of Theorem \ref{threenew}.
The quantity appearing in \eqref{equ:rapid} and \eqref{equ:quasi} is related to the  the \emph{Lin--Strauss Morawetz inequality} of \cite{LinStr}, which is given by
	\begin{equation}\label{eq:ls}
		\iint_{I\times\R^4}\frac{\vert u(t,x)\vert^{p+2}}{\vert x\vert}\,dx\,dt\lesssim
		\norm{\vert\nabla\vert^{1/2}u}_{L_t^\infty L_x^2(I\times\R^4)}^2.
		\end{equation}
Due to the weight $\tfrac{1}{\vert x\vert}$, the Lin--Strauss Morawetz inequality is well suited for preventing concentration near the origin, and hence it is most effective in the radial setting. In fact, it is the use of Lin--Strauss Morawetz inequality that leads to the restriction to the radial setting in Theorem~\ref{theorem}. We cannot use this estimate directly, however, as the solutions we consider need only belong to $L_t^\infty\dot{H}_x^{s_c}$ (and so the right-hand side of \eqref{eq:ls} need be infinite).

A further manifestation of the minimality of $u$ as a blow-up
solution is the absence of the scattered wave at the endpoints of
the lifespan $I$; more formally, we have the following Duhamel
formula, which is important for
showing the additional
decay and negative regularity for the rapid frequency cascade. This
is a robust consequence of almost periodicity modulo symmetries;
see, for example, \cite{CKSTT07}.
\begin{lemma} [Reduced Duhamel formula]\label{nowaste} Let $u: [0,T_{max})\times\R^4\to\C$ be a maximal-lifespan almost periodic solution to \eqref{equ1.1}. Then for all $t\in [0,T_{\rm max})$ we have
\begin{equation}\label{nwd}
u(t)=\lim_{T\nearrow T_{\rm max}}i\int_t^{T}e^{i(t-s)\Delta}F(u(s))\,ds
\end{equation}
as a weak limit in $\dot{H}_x^{s_c}(\R^4)$.
\end{lemma}

We will use this lemma to show that the almost periodic
solution enjoys the ``long time Strichartz estimate" \`a la
Dodson \cite{Dodson3}, see Theorem \ref{lse} below.
We will utilize this ``long time Strichartz estimate" and the no-waste Duhamel formula to
show the rapid frequency cascade solutions admit the negative
regularity property. Then making use of a similar method used in
\cite{KTV2009,KV2010,KVZ2008}, we can also show that the mass of the
rapid frequency cascade solution is zero and so we get a
contradiction.

 Finally, to preclude the quasi-soliton solutions, one can show that it admits additional decay in the sprit of no-waste
Duhamel formula and the compactness, see \cite[Proposition 3.1]{MMZ}.

\begin{proposition}[Additional decay, \cite{MMZ}]\label{adecay1} Let $s_c>1$. Suppose $u:[0,T_{\rm max})\times\R^4\to\C$ is an almost periodic solution to \eqref{equ1.1} such that $\inf_{t\in[0,T_{\rm max})}N(t)\geq 1$. Then
        \begin{equation}u\in L_t^\infty L_x^{p+1}([0,T_{\rm max})\times\R^4).
        \label{breaking scaling}
        \end{equation}
\end{proposition}

This together with long-time Strichartz estimate allows us to show
the localized Morawetz inequality enjoys the bounds as
follows.

\begin{proposition}\label{uplowbound}  Let $u$ be a quasi-soliton solution to \eqref{equ1.1} in the sense of Theorem
\ref{threenew}. Then for any $\eta>0$, there exists $N_0=N_0(\eta)$
such that for all $N\leq N_0$
\begin{equation*}
\begin{split}
K_I\lesssim_u\int_I\int_{\R^4}\frac{|u_{\geq
N}(t,x)|^{p+2}}{|x|}\;dx\;dt
\lesssim_u\eta\big(N^{1-2s_c}+K_I\big),
\end{split}
\end{equation*}
for all $I\subset[0,+\infty),$ where $K_I:=\int_IN(t)^{3-2s_c}dt.$
\end{proposition}

This leads to a contradiction in the case when $\int_0^\infty
N(t)^{3-2s_c}=+\infty,$ since $K_I$ can be taken arbitrarily large.

The paper is organized as follows. In Section $2$, we give some
useful lemmas. In Section $3$, we show the long time Strichartz
estimate. In Section $4$, we exclude the existence of rapid  frequency cascade solutions in the sense of
Theorem \ref{threenew}. In Section $5$, we
show the good upper bound for the localized Morawetz
estimate. Finally, we rule out the existence of quasi-soliton solution. Hence we conclude the proof of Theorem \ref{theorem}.

\subsection*{Acknowledgements} The authors would like to thank Jason Murphy for his helpful discussions.
This work was partially supported by the ERC-2014-CoG, 646650: SingWave.




\section{Preliminaries}

\subsection{Some notation}
For nonnegative quantities $X$ and $Y$, we will write $X\lesssim Y$ to denote the estimate $X\leq C Y$ for some $C>0$. If $X\lesssim Y\lesssim X$, we will write $X\sim Y$. Dependence of implicit constants on the power $p$ or the dimension will be suppressed; dependence on additional parameters will be indicated by subscripts. For example, $X\lesssim_u Y$ indicates $X\leq CY$ for some $C=C(u)$.

We will use the expression $\text{\O}(X)$ to denote a finite linear combination of terms
 that resemble $X$ up to Littlewood--Paley projections, complex conjugation, and/or maximal functions.
  For example,  We will use the expression $X\pm$ to denote $X\pm\eps$ for any $\eps>0$.

For a spacetime slab $I\times\R^4$, we write $L_t^q L_x^r(I\times\R^4)$ for the Banach space of functions $u:I\times\R^4\to\C$ equipped with the norm
    $$\|u\|_{L_t^{q}L_x^{r}(I\times\R^4)}:=\bigg(\int_I \|u(t)\|_{L_x^r(\R^4)}\bigg)^{1/q},$$
with the usual adjustments when $q$ or $r$ is infinity. When $q=r$, we abbreviate $L_t^qL_x^q=L_{t,x}^q$.
We will also often abbreviate $\norm{f}_{L_x^r(\R^4)}$ to $\norm{f}_{L_x^r}.$ For $1\leq r\leq\infty$,
we use $r'$ to denote the dual exponent to $r$, i.e. the solution to $\tfrac{1}{r}+\tfrac{1}{r'}=1.$

We define the Fourier transform on $\mathbb{R}^4$ by
\begin{equation*}
\aligned \widehat{f}(\xi):= \tfrac{1}{4\pi^2}\int_{\mathbb{R}^4}e^{- ix\cdot \xi}f(x)\,dx.
\endaligned
\end{equation*}
We can then define the fractional differentiation operator $\vert\nabla\vert^s$ for $s\in \R$ via
    $$\widehat{\vert\nabla\vert^s f}(\xi):=\vert\xi\vert^s\hat{f}(\xi),$$
with the corresponding homogeneous Sobolev norm
    $$\norm{f}_{\dot{H}_x^s(\R^4)}:=\norm{\vert\nabla\vert^s f}_{L_x^2(\R^4)}.$$

\subsection{Basic harmonic analysis}

We will make frequent use of the Littlewood--Paley projection operators.
Specifically, we let $\varphi$ be a radial bump function supported
on the ball $|\xi|\leq 2$ and equal to 1 on the ball $|\xi|\leq 1$.
For $N\in 2^{\mathbb{Z}}$, we define the Littlewood--Paley projection operators by
\begin{align*}
&\widehat{P_{\leq N}f}(\xi) := \mathcal{F}(f_{\leq N})(\xi):= \varphi(\xi/N)\widehat{f}(\xi),
 \\ &\widehat{P_{> N}f}(\xi) :=  \mathcal{F}(f_{>N})(\xi) :=
(1-\varphi(\xi/N))\widehat{f}(\xi),
\\ &\widehat{P_{N}f}(\xi) :=  \mathcal{F}(f_N)(\xi):=
(\varphi(\xi/N)-\varphi({2\xi}/{N}))\widehat{f}(\xi).
\end{align*}
We may also define $$P_{M<\cdot\leq N}:=P_{\leq N}-P_{\leq M}=\sum_{M<N'\leq N}P_N'$$
for $M<N.$ All such summations should be understood to be over $N'\in 2^{\mathbb{Z}}.$

The Littlewood--Paley operators commute with derivative operators,
the free propagator, and the conjugation operation. These operators are
self-adjoint and bounded on every $L^p_x$ and $\dot{H}^s_x$ space
for $1\leq p\leq \infty$ and $s\geq 0$. They also obey the
following standard Bernstein estimates:
 For $1\leq r\leq q\leq\infty$ and $s\geq 0$,
\begin{align*}
\big\||\nabla|^{\pm s} P_{N} f \big\|_{L_x^r(\R^4)} & \thicksim  N^{\pm s}
\big\|
P_{N} f \big\|_{L_x^r(\R^4)},  \\
\big\||\nabla|^s P_{\leq N} f \big\|_{L_x^r(\R^4)} & \lesssim   N^{s}
\big\|
P_{\leq N} f \big\|_{L_x^r(\R^4)},  \\
 \big\| P_{> N} f \big\|_{L_x^r(\R^4)} & \lesssim  N^{-s} \big\|
|\nabla|^{s}P_{> N} f \big\|_{L_x^r(\R^4)}, \\
\big\| P_{\leq N} f \big\|_{L^q(\R^4)} & \lesssim
N^{\frac{4}{r}-\frac{4}{q}} \big\|
P_{\leq N} f \big\|_{L_x^r(\R^4)}.
\end{align*}

We will need the following fractional calculus estimates and paraproduct estimates.

\begin{lemma} $(i)$ $($Fractional product rule, \cite{CW}$)$
Let $s\geq0$, and $1<r,r_j,q_j<\infty$ satisfy
$\frac1r=\frac1{r_i}+\frac1{q_i}$ for $i=1,2$. Then
\begin{equation}\label{moser}
\big\||\nabla|^s(fg)\big\|_{L_x^r(\R^4)}\lesssim\|f\|_{{L_x^{r_1}(\R^4)}}\big\||\nabla|^sg
\big\|_{{L_x^{q_1}(\R^4)}}+\big\||\nabla|^sf\big\|_{{L_x^{r_2}(\R^4)}}\|g\|_{{L_x^{q_2}(\R^4)}}.
\end{equation}

$(ii)$ $($Paraproduct estimate, \cite{MMZ}$)$
 Let $0<s<1$. If $1<r<r_1<\infty$ and $1<r_2<\infty$ satisfy
        $\tfrac{1}{r_1}+\tfrac{1}{r_2}=\tfrac{1}{r}+\tfrac{s}{4}<1,$
then
        \begin{equation}\label{pde1}
        \norm{\vert\nabla\vert^{-s}(fg)}_{L_x^r(\R^4)}\lesssim
        \norm{\vert\nabla\vert^{-s}f}_{L_x^{r_1}(\R^4)}\norm{\vert\nabla\vert^{s}g}_{L_x^{r_2}(\R^4)}.
        \end{equation}

$(iii)$ $($Basic estimate$)$ From the Hardy inequality and interpolation, we easily get for $0\leq s\leq1$
\begin{equation}\label{equ:hardy}
\big\||\nabla|^s\big(\tfrac{x}{|x|}u\big)\big\|_{L_x^2(\R^4)}\lesssim\big\||\nabla|^su\big\|_{L^2_x(\R^4)}.
\end{equation}
\end{lemma}

 \subsection{Strichartz estimates}\label{sze}
 Let $e^{it\Delta}$ be the free Schr\"odinger propagator, given by
    \begin{equation}\label{explicit formula}
    [e^{it\Delta}f](x)=\tfrac{-1}{16\pi^2 t^2}\int_{\R^4} e^{i\vert x-y\vert^2/4t}f(y)\,dy
    \end{equation}
for $t\neq 0$. From this explicit formula we can read off the dispersive estimate
    $$\norm{e^{it\Delta}f}_{L_x^\infty(\R^4)}\lesssim\vert t\vert^{-2}\norm{f}_{L_x^1(\R^4)}$$
for $t\neq 0$. Interpolating with $\norm{e^{it\Delta}f}_{L_x^2(\R^4)}\equiv \norm{f}_{L_x^2(\R^4)}$ then yields
\begin{equation}\label{dispers}
\big\|e^{it\Delta}f \big\|_{L^q_x(\R^4)} \leq C|t|^{-2(1-\frac2{q})}
\|f\|_{L^{q'}_x(\R^4)}
\end{equation}
for $t\neq 0$ and $2\leq r\leq\infty$, where $\tfrac{1}{q}+\tfrac{1}{q'}=1.$ This estimate implies the standard Strichartz estimates, which we will state below. First, we need to make the following definition:

\begin{definition}[Admissible pairs]\label{def1} A pair of exponents $(q,r)$ is called \emph{Schr\"odinger admissible} if $2\leq q,r\leq\infty$
and $\tfrac{2}{q}+\tfrac{4}{r}=2.$ For a spacetime slab $I\times\R^4$, we define the Strichartz norm
    $$\norm{u}_{S^0(I)}:=\sup\big\{\|u\|_{L_t^{q}L_x^{r}(I\times\R^4)}:(q,r)\text{ Schr\"odinger admissible}\big\}.$$
We denote $S^0(I)$ to be the closure of all
test functions under this norm and write $N^0(I)$ for the dual of
$S^0(I)$.
\end{definition}

We may now state the standard Strichartz estimates in the form that we will need them.

\begin{proposition}[Strichartz \cite{GV, KeT98, St}]\label{prop1}
Let $s\geq 0$ and suppose $u:I\times\R^4\to\C$ is a solution to $(i\partial_t+\Delta)u=F$. Then
    $$\norm{\vert\nabla\vert^s u}_{S^0(I)}\lesssim\norm{\vert\nabla\vert^s u(t_0)}_{L_x^2(\R^4)}+\norm{\vert\nabla\vert^s F}_{N^0(I)}$$
for any $t_0\in I$.

\end{proposition}

\section{Long time Strichartz estimate}
In this section we establish a long-time Strichartz estimate for almost periodic solutions to \eqref{equ1.1} as in Theorem \ref{threenew}.

\begin{theorem}[Long-time Strichartz estimate]\label{lse}
Let $s_c>3/2,$ and let $u:[0,T_{\rm max})\times\R^4\to\C $ be an almost
periodic solution to \eqref{equ1.1} with $N(t)\equiv N_k\geq1$ on
each characteristic $J_k\subset[0,T_{\rm max}).$  Suppose
\begin{equation}\label{equ:assuadi}
u\in L_t^\infty([0,T_{\rm max}), \dot{H}^s(\R^4)),\quad\text{for some}\quad s_c-\tfrac12<s\leq s_c.
\end{equation}
Then on
any compact time interval $I\subset[0,T_{max}),$ which is a union of
characteristic subintervals $J_k$, and for any $N>0,$ we
have
\begin{equation}\label{lse1}
\big\||\nabla|^{s_c}u_{\leq
N}\big\|_{L_t^2L_x^4(I\times\R^4)}\lesssim_u1+N^{\sigma(s)}K_I^{1/2},
\end{equation}
where $K_I:=\int_IN(t)^{3-2s_c}\,dt$ and $\sigma(s)=2s_c-s-\frac12$. In particular, for $s=s_c$, we have
\begin{equation}\label{lse2}
\big\||\nabla|^{s_c}u_{\leq
N}\big\|_{L_t^2L_x^4(I\times\R^4)}\lesssim_u1+N^{s_c-\frac12}K_I^{1/2}.
\end{equation}
 Moreover, for any $\eta>0$, there
exists $N_0=N_0(\eta)$ such that for all $N\leq N_0,$
\begin{equation}\label{lsesmall}
\big\||\nabla|^{s_c}u_{\leq
N}\big\|_{L_t^2L_x^4(I\times\R^4)}\lesssim_u\eta\big(1+N^{s_c-1/2}K_I^{1/2}\big).
\end{equation}
Furthermore, the implicit constants in \eqref{lse1} and
\eqref{lsesmall} are independent of $I$.
\end{theorem}

We fix $I\subset[0,T_{\rm max})$ to be a union of contiguous characteristic subintervals. Throughout the proof,
all spacetime norms will be taken over $I\times\R^4$ unless explicitly stated otherwise. For $N>0$, we define the quantities
$$B(N):=\big\||\nabla|^{s_c}u_{\leq
N}\big\|_{L_t^2L_x^4(I\times\R^4)}\quad\text{and}
\quad B_k(N):=\big\||\nabla|^{s_c}u_{\leq
N}\big\|_{L_t^2L_x^4(J_k\times\R^4)}.$$

We will prove Theorem \ref{lse} by induction.  For the base case, we have the following

\begin{lemma}\label{base}
The estimate \eqref{lse1} holds for $N\geq N_{\rm max}:=\sup\limits_{J_k\subset
I}N_k.$
\end{lemma}
\begin{proof}
 We have by Lemma \ref{fspssn}
\begin{align*}
\big\||\nabla|^{s_c}u_{\leq
N}\big\|_{L_t^2L_x^4}^2&\lesssim_u1+\int_I N(t)^2\,dt
 \\ &\lesssim_u1+N_{\rm max}^{2s_c-1}\int_IN(t)^{3-2s_c}\,dt
 \\ &\lesssim1+N^{2\sigma(s)}K_I,
\end{align*}
which shows \eqref{lse1}.
\end{proof}

To complete the induction, we establish a recurrence relation for $B(N)$. To this end, we first let $\eps_0>0$ and $\eps>0$ be small parameters to be determined later. We use the compact property to find $c=c(\eps)$ so that
\begin{equation}\label{xcx}
\big\||\nabla|^{s_c}u_{\leq cN(t)}\big\|_{L_t^\infty L_x^2}<
\varepsilon.
\end{equation}
The recurrence relations we will use take the following form.

\begin{lemma}[Recurrence relations for $B(N)$]\label{rfthm}
\begin{equation}\label{rf1}
\begin{split}
B(N)\lesssim_u&\inf_{t\in I}\big\||\nabla|^{s_c}u_{\leq
N}(t)\big\|_{L_x^2}+C(\varepsilon,\varepsilon_0)N^{\sigma(s)}K_I^{1/2}+\varepsilon^pB(N/\varepsilon_0)\\
&+\sum_{M>N/\varepsilon_0}(\tfrac{N}{M})^{s_c}B(M)
\end{split}
\end{equation}
uniformly in $N$ for some positive constant
$C(\varepsilon,\varepsilon_0)$.

We also have the following refinement of \eqref{rf1}:
\begin{equation}\label{rf2}
\begin{split}
B(N)\lesssim_ug(N)\big(1+N^{s_c-1/2}K_I^{1/2}\big)+\varepsilon^pB(N/\varepsilon_0)+\sum_{M>N/\varepsilon_0}(\tfrac{N}{M})^{s_c}B(M),
\end{split}
\end{equation}
where
\begin{equation}\label{g(N)}
g(N):=\inf_{t\in I}\big\||\nabla|^{s_c}u_{\leq N}(t)\big\|_{
L_x^2}+C(\varepsilon,\varepsilon_0)\sup_{J_k\subset
I}\big\||\nabla|^{s_c}u_{\leq N/\varepsilon_0}\big\|_{L_t^\infty
L_x^2(J_k\times\R^4)}.
\end{equation}
\end{lemma}

\begin{proof}
First, by  Strichartz estimate, we get
\begin{equation}\label{nte}
B(N)\lesssim \inf_{t\in I}\big\|{\nsc u_{\leq N}(t)}\big\|_{L_x^2}+\big\||\nabla|^{s_c}P_{\leq
N}F(u)\big\|_{L_t^2L_x^{4/3}}.
\end{equation}
Hence we only need to estimate the nonlinear term. To this end, we decompose
$$F(u)=F(u_{\leq N/\varepsilon_0})+\big(F(u)-F(u_{\leq N/\varepsilon_0})\big).$$
Using Bernstein, H\"older, Sobolev embedding, we obtain
\begin{align}\nonumber
\big\||\nabla|^{s_c}P_{\leq
N}\big(F(u)-F(u_{\leq N/\varepsilon_0})\big)\big\|_{L_t^2L_x^{4/3}}\lesssim&N^{s_c}\big\|F(u)-F(u_{\leq N/\varepsilon_0})\big\|_{L_t^2L_x^{4/3}}\\\nonumber
\lesssim&N^{s_c}\|u\|^p_{L_t^\infty L_x^{2p}}\|u_{>N/\varepsilon_0}\|_{L_t^2L_x^4}\\\label{equ:term1}
\lesssim&\sum_{M>N/\varepsilon_0}\big(\tfrac{N}{M}\big)^{s_c}B(M).
\end{align}

To estimate the term $\big\||\nabla|^{s_c}P_{\leq
N}F(u_{\leq N/\varepsilon_0})\big\|_{L_t^2L_x^{4/3}},$ we restrict  our attention to a single characteristic interval $J_k$. It is easy to see that
\begin{align}\nonumber
&\big\||\nabla|^{s_c}P_{\leq
N}F(u_{\leq N/\varepsilon_0})\big\|_{L_t^2L_x^{4/3}(J_k\times\R^4)}\\\nonumber
\lesssim&\|u_{\leq N/\varepsilon_0}\|_{L_t^\infty L_x^{2p}(J_k\times\R^4)}^p
\big\|\nsc u_{\leq N/\varepsilon_0}\big\|_{L_t^2L_x^4(J_k\times\R^4)}\\\label{equ:term21}
\lesssim&\big\|\nsc P_{\leq c(\varepsilon)N_k}u_{\leq N/\varepsilon_0}\|_{L_t^\infty L_x^{2}(J_k\times\R^4)}^p\big\|\nsc u_{\leq N/\varepsilon_0}\big\|_{L_t^2L_x^4(J_k\times\R^4)}\\\label{equ:term22}
&+\big\|\nsc P_{> c(\varepsilon)N_k}u_{\leq N/\varepsilon_0}\|_{L_t^\infty L_x^{2}(J_k\times\R^4)}^p\big\|\nsc u_{\leq N/\varepsilon_0}\big\|_{L_t^2L_x^4(J_k\times\R^4)}.
\end{align}
By \eqref{xcx}, we have
$$\eqref{equ:term21}\lesssim \varepsilon^p\big\|\nsc u_{\leq N/\varepsilon_0}\big\|_{L_t^2L_x^4(J_k\times\R^4)}.$$
Note that it suffices to consider the case $c(\varepsilon)N_k\leq {N}/{\varepsilon_0}$ in \eqref{equ:term22}, i.e $\frac{N}{c(\varepsilon)N_k\varepsilon_0}\geq1$. In this case, using H\"older, Bernstein, Sobolev embedding, interpolation, Lemma \ref{fspssn}, and \eqref{assume1.1}, we derive that
\begin{align*}
\eqref{equ:term22}\lesssim&\big(\tfrac{N}{c(\varepsilon)N_k\varepsilon_0}\big)^{s_c-\frac12}\big\|\nsc u_{\leq N/\varepsilon_0}\|_{L_t^\infty L_x^{2}(J_k\times\R^4)}\\
\lesssim&\big(\tfrac{N}{c(\varepsilon)N_k\varepsilon_0}\big)^{s_c-\frac12}\big(\tfrac{N}{\varepsilon_0}\big)^{s_c-s}\|u_{\leq N/\varepsilon_0}\|_{L_t^\infty \dot{H}^s(J_k\times\R^4)}.
\end{align*}
Hence,
\begin{align*}
&\big\||\nabla|^{s_c}P_{\leq
N}F(u_{\leq N/\varepsilon_0})\big\|_{L_t^2L_x^{4/3}(I\times\R^4)}^2\\
\lesssim&\sum_{J_k\subset I}\big\||\nabla|^{s_c}P_{\leq
N}F(u_{\leq N/\varepsilon_0})\big\|_{L_t^2L_x^{4/3}(J_k\times\R^4)}^2\\
\lesssim&\sum_{J_k\subset I}\Big[\varepsilon^{2p}\big\|\nsc u_{\leq N/\varepsilon_0}\big\|_{L_t^2L_x^4(J_k\times\R^4)}^2
+\big(\tfrac{N}{cN_k\varepsilon_0}\big)^{2(s_c-\frac12)}\big(\tfrac{N}{\varepsilon_0}\big)^{2(s_c-s)}\|u_{\leq N/\varepsilon_0}\|_{L_t^\infty \dot{H}^s(J_k\times\R^4)}^2\Big]\\
\lesssim&\varepsilon^{2p}B(N/\varepsilon_0)+c^{-2(s_c-\frac12)}\varepsilon_0^{-2\sigma(s)}
\sup_{J_k\subset I}\|u_{\leq N/\varepsilon_0}\|_{L_t^\infty \dot{H}^s(J_k\times\R^4)}^2N^{2\sigma(s)}K_I.
\end{align*}
This together with \eqref{equ:term1} implies Lemma \ref{rfthm}.

\end{proof}

Next, we turn to prove Theorem \ref{lse}.
\begin{proof}[Proof of Theorem \ref{lse}] From Lemma \ref{base}, we see that \eqref{lse1} holds for
$N\geq N_{\rm max}$. That is, we have
\begin{equation}\label{xhdd}
B(N)\leq C(u)\big(1+N^{\sigma(s)}K_I^{1/2}\big),
\end{equation}
for all $N\geq N_{\rm max}$. Clearly this
inequality remains true if we replace $C(u)$ by any larger constant.

We now suppose that \eqref{xhdd} holds for frequencies above $N$ and use the recurrence formula \eqref{rf1} to show that \eqref{xhdd} holds at frequency $N/2$. Choosing $\eps_0<1/2$, we use \eqref{rf1} and \eqref{xhdd} to obtain
\begin{align*}
B\big(N/2\big)
    &\leq \tilde{C}(u)\Big(1+C(\varepsilon,\varepsilon_0)\big(N/2\big)^{\sigma(s)}K_I^{1/2}+\varepsilon^pB(N/2\varepsilon_0)
    +\sum_{M>N/2\varepsilon_0}(\tfrac{N}{2M})^{s_c}B(M)\Big)
    \\ &\leq\tilde{C}(u)\Big(1+C(\varepsilon,\varepsilon_0)\big(N/2\big)^{\sigma(s)}K_I^{1/2}+\varepsilon^p
    \big(1+(N/2\varepsilon_0)^{\sigma(s)}K_I^\frac12\big)\\
    &\qquad\qquad+C(u)\sum_{M>N/2\varepsilon_0}(\tfrac{N}{2M})^{s_c}\big(1+M^{\sigma(s)}K_I^{1/2}\big)\Big)
    \\ &\leq\tilde{C}(u)\Big(1+C(\varepsilon,\varepsilon_0)\big(N/2\big)^{\sigma(s)}K_I^{1/2}+\varepsilon^p
    \big(1+(N/2\varepsilon_0)^{\sigma(s)}K_I^\frac12\big)\\
    &\qquad\qquad
    +C(u)\varepsilon_0^{s_c}+C(u)\varepsilon_0^{s_c-\sigma(s)}(N/2)^{\sigma(s)}K_I^\frac12\Big)
    \\&=\tilde{C}(u)\Big(1+\varepsilon^p+C(u)\varepsilon_0^{s_c}\Big)\\
    &\quad+\tilde{C}(u)\big(C(\varepsilon,\varepsilon_0)+\varepsilon^p\varepsilon_0^{-\sigma(s)}
+C(u)\varepsilon_0^{s_c-\sigma(s)}\big)(N/2)^{\sigma(s)}K_I^{1/2},
\end{align*}
where we use $s_c-\frac12<s$ in the third inequality to guarantee the
convergence of the sum.

If we now choose $\varepsilon_0$ possibly even smaller depending on $\tilde{C}(u)$; $\varepsilon$ sufficiently small depending on $\tilde{C}(u)$ and $\varepsilon_0$; and $C(u)$ possibly larger such that
$$C(u)\geq\tilde{C}(u)\Big(1+\varepsilon^p+C(u)\varepsilon_0^{s_c}\Big)+\tilde{C}(u)
\big(C(\varepsilon,\varepsilon_0)+\varepsilon^p\varepsilon_0^{-\sigma(s)}
+C(u)\varepsilon_0^{s_c-\sigma(s)}\big),$$ we get
\begin{align*}
B\big(N/2\big)
         &\leq C(u)\Big(1+(N/2)^{\sigma(s)}K_I^{1/2}\Big).
\end{align*}
Thus \eqref{xhdd} holds at frequency ${N}/2$, and so we conclude \eqref{lse1} by induction.

Next, we will use the recurrence formula \eqref{rf2} to show \eqref{lsesmall}. To do this, we note that for fixed $\eps,\eps_0>0$, we can use the compact property and the fact that $\inf_{t\in I}N(t)\geq1$ to get
        \begin{equation}\label{g(N) to zero}
        \lim_{N\to0}g(N)=0,
        \end{equation}
where $g(N)$ is as in \eqref{g(N)}.

\end{proof}

\section{The rapid frequency-cascade scenario}

In this section, we rule out the existence of rapid frequency-cascade solutions, that is, almost periodic solutions as in Theorem~\ref{threenew} such that $\int_0^{T_{\rm max}} N(t)^{3-2s_c}\,dt<\infty.$ The proof will rely primarily on the long-time Strichartz estimate proved in the previous section.

\begin{theorem}[No rapid frequency-cascades]\label{no frequency-cascades} Let $s_c>3/2.$ Then there are no radial almost periodic solutions $u:[0,T_{\rm max})\times\R^4\to\C$ to \eqref{equ1.1} with $N(t)\equiv N_k\geq 1$ on each characteristic subinterval $J_k\subset[0,T_{\rm max})$ such that
    \begin{equation}\label{rf blowup}
    \|u\|_{L_{t,x}^{3p}([0,T_{\rm max})\times\R^4)}=+\infty
    \end{equation}
and
    \begin{equation}\label{rf K}
    K:=\int_0^{T_{\rm max}} N(t)^{3-2s_c}\,dt<+\infty.
    \end{equation}
\end{theorem}

We argue by contradiction. Suppose that $u$ were such a solution. Then, using \eqref{rf K} and Corollary \ref{constancy cor}, we see
\begin{equation}
\lim_{t\to T_{\rm max}}N(t)=\infty,
\end{equation}
whether $T_{\rm max}$ is finite or infinite. Combining this with the compact property, we see that
\begin{equation}\label{gtx}
\lim_{t\to T_{\rm max}}\big\||\nabla|^{s_c}u_{\leq
N}(t)\big\|_{L_x^2(\R^4)}=0
\end{equation}
for any $N>0$.

\begin{proposition}[Lower regularity]\label{prop:lowreg}
Let $s_c>3/2.$ Let $u:[0,T_{\rm max})\times\R^4\to\C$ be an almost periodic solution with \eqref{rf K}. Suppose
\begin{equation}\label{equ:assumsc}
u\in L_t^\infty([0,T_{\rm max}),\dot{H}^s(\R^4))\quad\text{for some}\quad s_c-\tfrac12<s\leq s_c.
\end{equation}
Then,
\begin{equation}\label{equ:lowregu}
u\in L_t^\infty([0,T_{\rm max}),\dot{H}^\alpha(\R^4)),\quad \forall~s-\sigma(s)<\alpha\leq s_c,
\end{equation}
where $\sigma(s)=2s_c-s-\frac12.$
\end{proposition}

\begin{proof}

Let $I_n\subset [0,T_{\rm max})$ be a nested sequence of compact time intervals, each of which is a contiguous union of characteristic subintervals. We claim that for any $N>0$, we have
\begin{equation}\label{gxx}
\big\||\nabla|^{s_c}u_{\leq N}\big\|_{L_t^2L_x^4(I_n\times\R^4)}\lesssim_u\inf_{t\in I_n}
\big\||\nabla|^{s_c}u_{\leq N}(t)\big\|_{L_x^2}+N^{\sigma(s)}.
\end{equation}
Indeed, defining
        $$B_n(N):=\big\||\nabla|^{s_c}u_{\leq N}\big\|_{L_t^2L_x^4(I_n\times\R^4)},$$
we have by \eqref{rf1} and \eqref{rf K}
\begin{align*}
B_n(N) \lesssim_u\inf_{t\in I_n}\big\||\nabla|^{s_c}u_{\leq N}(t)\big\|_{L_x^2}+C(\varepsilon,\varepsilon_0)N^{\sigma(s)}+\varepsilon^pB(N/\varepsilon_0)+\sum_{M>N/\varepsilon_0}(\tfrac{N}{M})^{s_c}B_n(M).
\end{align*}
Arguing as we did to obtain \eqref{lse1}, we derive \eqref{gxx}.

Now, letting $n\to\infty$ in \eqref{gxx} and using \eqref{gtx}, we obtain
\begin{equation}\label{rfc11}
\big\||\nabla|^{s_c}u_{\leq
N}\big\|_{L_t^2L_x^4([0,T_{\rm max})\times\R^4)}\lesssim_u
N^{\sigma(s)}
\end{equation}
for all $N>0$.

We now claim that \eqref{rfc11} implies
\begin{equation}\label{ngtcl}
\big\||\nabla|^{s}u_{\leq N}\big\|_{L_t^\infty
L_x^2([0,T_{\rm max})\times\R^4)}\lesssim_u N^{2s_c-1/2}
\end{equation}
for all $N>0.$ To this end, we first use the reduced Duhamel formula and Strichartz to get
\begin{align}
\big\||\nabla|^{s}u_{\leq N}\big\|_{L_t^\infty
L_x^2([0,T_{\rm max})\times\R^4)}\lesssim
\label{gfues}
\big\||\nabla|^{s}P_{\leq N}F(u)\big\|_{L_t^2L_x^{4/3}([0,T_{\rm max})\times\R^4)}.
\end{align}
%
%

We now decompose $F(u)$ by
\begin{align}\label{fgu}
F(u)=F(u_{\leq N})+\big(F(u)-F(u_{\leq N})\big)
\end{align}
and estimate the contribution of each piece individually.

We begin by estimating the contribution of the first term in \eqref{fgu}
to \eqref{gfues}. Using H\"older, the fractional product rule, the fractional chain rule, Sobolev embedding, interpolation, \eqref{assume1.1}, \eqref{equ:assumsc} and \eqref{rfc11}, we estimate
\begin{align}\nonumber\big\||\nabla|^{s}P_{\leq N}F(u_{\leq N})\big\|_{L_t^2 L_x^{4/3}}
\lesssim&\|u\|_{L_t^\infty L_x^\frac{4p}{s_c-s+2}}^p\big\||\nabla|^s u_{\leq N}\big\|_{L_t^2 L_x^\frac{4}{2/p+s-1}}
\\\nonumber
\lesssim&\|u\|_{L_t^\infty \dot{H}^{s_1}}\big\|\nsc u_{\leq N}\big\|_{L_t^2 L_x^4}\\\nonumber
\lesssim&\big\|\nsc u_{\leq N}\big\|_{L_t^2 L_x^4}
\\\label{equ:firass}
\lesssim_u&N^{\sigma(s)},
\end{align}
where $s_1=2-\frac{s_c-s+2}{p}\in(s,s_c].$

Next, we estimate the contribution of the second term in
\eqref{fgu} to \eqref{gfues}. We can use H\"older, Bernstein, \eqref{pde1}, \eqref{equ:assumsc} and \eqref{rfc11} to estimate
\begin{align*}
&\quad\big\||\nabla|^{s}P_{\leq N}\big(F(u)-F(u_{\leq N})\big)\big\|_{L_t^2
L_x^{4/3}}
\\ &\lesssim
N^{s_c}\Big\||\nabla|^{-(s_c-s)}\int_0^1F'(u_{\leq N}+\lambda u_{>N})u_{>N}\;d\lambda\Big\|_{L_t^2
L_x^{4/3}}
\\
&\lesssim \int_0^1N^{s_c}\big\||\nabla|^{s_c-s}F'(u_{\leq N}+\lambda u_{>N})\big\|_{L_t^\infty
L_x^\frac{2}{s_c-s+1}}\big\||\nabla|^{-(s_c-s)}u_{>N}\big\|_{L_t^2L_x^\frac{4}{2/p+s-1}}\;d\lambda
\\
&\lesssim N^{s_c} \|u\|_{L_t^\infty L_x^{\frac{4p}{s_c-s+2}}}^{p-1}\big\||\nabla|^{s_c-s}u\big\|_{L_t^\infty
L_x^\frac{4p}{(p+1)(s_c-s)+2}}\big\|u_{>N}\big\|_{L_t^2L_x^4}
\\
&\lesssim\|u\|_{L_t^\infty \dot{H}^{s_1}}^{p} \sum_{M>N}\big(\tfrac{N}{M}\big)^{s_c}\big\|\nsc u_M\big\|_{L_t^2L_x^4}\\
&\lesssim_u\sum_{M>N}\big(\tfrac{N}{M}\big)^{s_c}M^{\sigma(s)}\\
&\lesssim_uN^{\sigma(s)},
\end{align*}
where $s_1=2-\frac{s_c-s+2}{p}\in(s,s_c].$
This together with \eqref{equ:firass} and \eqref{gfues} implies the claim \eqref{ngtcl}.

Finally, using Bernstein inequality, \eqref{rfc11}, and \eqref{ngtcl}, we obtain
\begin{align*}
\big\||\nabla|^\alpha  u\big\|_{L_t^\infty L_x^2}\lesssim&\big\||\nabla|^\alpha  u_{\geq1}\big\|_{L_t^\infty L_x^2}
+\sum_{N\leq 1}N^{\alpha-s}\big\||\nabla|^su_N\big\|_{L_t^\infty L_x^2}\\
\lesssim&\big\|\nsc u\big\|_{L_t^\infty L_x^2}+\sum_{N\leq 1}N^{\alpha+\sigma(s)}\\
\lesssim&1,
\end{align*}
where we need the restriction $s-\sigma(s)<\alpha\leq s_c.$

\end{proof}

\begin{theorem}[Negative Regularity]\label{regular}  Let
$u:[0,T_{\rm max})\times\R^4\to\C$ be a  solution to \eqref{equ1.1} which is almost periodic modulo symmetries in the
sense of Theorem \ref{threenew}  with \eqref{rf K}. And assume that
$\inf\limits_{t\in[0,T_{\rm max})}N(t)\geq1$, then  for any $0<\varepsilon<\tfrac12$
\begin{equation}\label{equ:negre}
u\in L_t^\infty([0,T_{\rm max});~\dot H^{-\varepsilon}_x(\R^4)).
\end{equation}
\end{theorem}

\begin{proof}
First, we use Proposition \ref{prop:lowreg} with $s=s_c$ to get
\begin{equation}\label{equ:lowregust1}
u\in L_t^\infty([0,T_{\rm max}),\dot{H}^\alpha(\R^4)),\quad \forall~\tfrac12<\alpha\leq s_c,
\end{equation}
Applying  Proposition \ref{prop:lowreg} with $s=\big(s_c-\tfrac12\big)_+$ again, we obtain \eqref{equ:negre}.

\end{proof}

Now, we turn to prove Theorem \ref{no frequency-cascades}. It follows from Theorem \ref{regular} that $u\in L_t^\infty([0,T_{\rm max}),\dot{H}^{-\varepsilon}(\R^4))$ with $0<\varepsilon<\tfrac12.$
 For $\eta>0$, we can interpolate this bound with \eqref{xiaoc} to get
\begin{equation}\nonumber
\int_{|\xi|\leq c(\eta) N(t)}\big|\hat{u}(t,\xi)\big|^2d\xi\lesssim_u
\eta^\frac{\varepsilon}{s_c+\varepsilon}.
\end{equation}
Hence, we obtain by Plancherel
\begin{align*}
M[u_0]=M[u(t)]&=\int_{|\xi|\leq
c(\eta) N(t)}\big|\hat{u}(t,\xi)\big|^2d\xi+\int_{|\xi|\geq
c(\eta) N(t)}\big|\hat{u}(t,\xi)\big|^2d\xi\\
&\lesssim_u\eta^\frac{\varepsilon}{s_c+\varepsilon}+\big(c(\eta)N(t)\big)^{-2s_c}\big\||\nabla|^{s_c}u\big\|_{L_t^\infty L_x^2}^2\\
&\lesssim_u\eta^\frac{\varepsilon}{s_c+\varepsilon}+\big(c(\eta)N(t)\big)^{-2s_c}.
\end{align*}
Choosing $\eta$ small, letting $t\to T_{\rm max}$, and \eqref{gtx}, we can deduce that $M(u_0)=0$. Thus, we obtain $u\equiv 0$, which contradicts \eqref{rf blowup}.

Therefore, we complete the proof of Theorem  \ref{no frequency-cascades}.

\section{The frequency-localized Morawetz inequality}\label{flim section}

In this section, we establish spacetime bounds for the high-frequency portions of almost periodic solutions to \eqref{equ1.1}. We will use these estimates in the next section to preclude the existence of quasi-soliton solutions in the sense of Theorem \ref{threenew}.

\begin{theorem}[Frequency-localized Morawetz inequality]\label{uppbd} Let $s_c>3/2$, and $u:[0,T_{\rm max})\times\R^4\to\C$ be an almost periodic solution to \eqref{equ1.1} such that $N(t)\equiv N_k\geq 1$ on each characteristic subinterval $J_k\subset[0,T_{\rm max})$. Let $I\subset[0,T_{\rm max})$ be a compact time interval, which is a contiguous union of characteristic subintervals $J_k$. Then for any $\eta>0$, there exists $N_0=N_0(\eta)$ such that for all $N\leq N_0$, we have
\begin{equation}\label{limemd}
\int_I\int_{\R^4}\frac{|u_{> N}(t,x)|^{p+2}}{|x|}\,dx\,dt \lesssim_u\eta\big(N^{1-2s_c}+K_I\big),
\end{equation}
where $K_I:=\int_IN(t)^{3-2s_c}dt$. Furthermore, $N_0$ and the implicit constants above are independent of the interval $I$.
\end{theorem}

We will use the following corollary to prove Theorem \ref{uppbd}.

\begin{lemma}[Low and high frequencies control]\label{lhc}let $u:[0,T_{\rm max})\times\R^4\to\C $ be an almost
periodic solution to \eqref{equ1.1} with $N(t)\equiv N_k\geq1$ on
each characteristic $J_k\subset[0,T_{\rm max}).$ Then on any compact
time interval $I\subset[0,T_{\rm max})$, which is a union of continuous
subintervals $J_k$, and for any frequency $N>0$, we have
\begin{equation}\label{hct}
\|u_{\geq N}\|_{L_t^2 L_x^4(I\times\R^4)}\lesssim_u
N^{-s_c}\big(1+N^{2s_c-1}K_I\big)^\frac12.
\end{equation}

For any $\eta>0$, there exists $N_1=N_1(\eta)$ such
that for all $N\leq N_1$, we have
\begin{equation}\label{hct1}
\big\||\nabla|^\frac12u_{\geq N}\big\|_{L_t^\infty L_x^2(I\times\R^4)}\lesssim_u\eta
N^{\frac12-s_c}.
\end{equation}

From \eqref{lsesmall}, we know that for any $\eta>0$, there exists $N_2=N_2(\eta)$ such
that for all $N\leq N_2$
\begin{equation}\label{sct}
\big\||\nabla|^{s_c}u_{\leq N}\big\|_{L_t^2
L_x^4(I\times\R^4)}\lesssim_u\eta\big(1+N^{2s_c-1}K_I\big)^\frac12.
\end{equation}

\end{lemma}

\begin{proof}  Throughout the proof, all spacetime norms will be taken over $I\times\R^4$.

By Bernstein's inequality and \eqref{lse2}, we have
\begin{align*}
\|u_{\geq N}\|_{L_t^2 L_x^4}\lesssim&\sum_{M\geq N}M^{-s_c}\big\|\nsc u_M\big\|_{L_t^2L_x^4}\\
\lesssim&\sum_{M\geq N}M^{-s_c}\big(1+M^{s_c-\frac12}K_I^\frac12\big)\\
\lesssim&N^{-s_c}+N^{-\frac12}K_I^\frac12\\
\lesssim&N^{-s_c}(1+N^{2s_c-1}K_I)^\frac12,
\end{align*}
which shows \eqref{hct}.

Using Remark \ref{remark:radial} $(ii)$ and the fact that $\inf\limits_{t\in I}N(t)\geq 1$, we may find $c(\eta)>0$ such that
\begin{equation}\nonumber
\big\||\nabla|^{s_c}u_{\leq c(\eta)}\big\|_{L_t^\infty
L_x^2}\leq \eta.
\end{equation}
Combining this estimate with Bernstein, we get
\begin{align*}
N^{s_c-\frac12}\big\||\nabla|^\frac12u_{\geq N}\big\|_{L_t^\infty L_x^2}&\lesssim
        N^{s_c-\frac12}\big\||\nabla|^\frac12u_{N\leq\cdot\leq c(\eta)}\big\|_{L_t^\infty L_x^2}+N^{s_c-\frac12}\big\||\nabla|^\frac12u_{\geq c(\eta)}\big\|_{L_t^\infty L_x^2}
        \\& \lesssim\big\||\nabla|^{s_c}u_{\leq c(\eta)}\big\|_{L_t^\infty L_x^2}+\tfrac{N^{s_c-\frac12}}{c(\eta)^{s_c-\frac12}}\big\||\nabla|^{s_c}u\big\|_{L_t^\infty L_x^2}
        \\ & \lesssim\eta + N^{s_c-\frac12}.
\end{align*}
Taking $N$ sufficiently small, we get \eqref{hct1}.

\end{proof}

{\bf The proof of Theorem \ref{uppbd}:} Throughout the proof, all spacetime norms will be taken over $I\times\R^4$.

Let $0<\eta\ll1$ and choose
	$$N<\min\{N_1(\eta), \eta^2 N_2(\eta^{2s_c})\},$$
where $N_1$ and $N_2$ are as in Lemma~\ref{lhc}. In particular, we note that \eqref{hct} gives
	\begin{equation}\label{still small1}
	\xonorm{u_{>N/\eta^2}}{2}{4}\lesssim_u \eta N^{-s_c}(1+N^{2s_c-1}K_I)^{1/2}.
	\end{equation}
Moreover, as $N/\eta^2< N_2(\eta^{2s_c})$, we can apply \eqref{sct} to get
	\begin{equation}\label{still small2}
	\xonorm{\nsc u_{\leq N/\eta^2}}{2}{4}\lesssim_u \eta(1+N^{2s_c-1}K_I)^{1/2}.
	\end{equation}

Define the Morawetz action
$$\text{Mor}(t)=2\,\Im\int_{\R^4}\frac{x}{\vert x\vert}\cdot\nabla u_{>N}(t,x)\bar{u}_{>N}(t,x)\,dx.$$
Since $(i\partial_t+\Delta)u_{>N}=P_{>N}\big(F(u)\big)$, we obtain
	\begin{equation}
	\nonumber
	\partial_t \text{Mor}(t)\gtrsim\int_{\R^4} \frac{x}{\vert x\vert}\cdot\{P_{>N}\big(F(u)\big),u_{>N}\}_P\,dx,
	\end{equation}
where the \emph{momentum bracket} $\{\cdot,\cdot\}_P$ is defined by $\{f,g\}_P:=\Re(f\nabla\bar{g}-g\nabla\bar{f}).$
Thus, by the fundamental theorem of calculus, we get
\begin{align}
\iint_{I\times\R^4}\frac{x}{\vert x\vert}\cdot\{P_{> N}\big(F(u)\big),u_{> N}\}_P\,dx\lesssim\norm{\text{Mor}}_{L_t^\infty(I)}.\label{flmor start}
\end{align}	

Noting that $\{F(u),u\}_P=-\tfrac{p}{p+2}\nabla(\vert u\vert^{p+2}),$ we write
	\begin{align*}
	\{P_{>N}\big(F(u)\big),u_{>N}\}_P
	&=\{F(u),u\}_P-\{F(u_{\leq N}),u_{\leq N}\}_P
	\\ &\quad-\{F(u)-F(u_{\leq N}),u_{\leq N}\}_P-\{P_{\leq N}\big(F(u)\big),u_{>N}\}_P
	\\ &=-\tfrac{p}{p+2}\nabla(\vert u\vert^{p+2}-\vert u_{\leq N}\vert^{p+2})
	-\{F(u)-F(u_{\leq N}),u_{\leq N}\}_P
	\\&\quad-\{P_{\leq N}\big(F(u)\big),u_{>N}\}_P
	\\&=:I+II+III.
	\end{align*}

Integrating by parts, we see that $I$ contributes to the left-hand side of \eqref{flmor start} a multiple of
	\begin{align*}
	\iint_{I\times\R^4}&\frac{\vert u_{>N}(t,x)\vert^{p+2}}{\vert x\vert}\,dx\,dt
	\end{align*}
and to the right-hand side of \eqref{flmor start} a multiple of
	\begin{equation}
	\big\|\tfrac{1}{\vert x\vert}(\vert u\vert^{p+2}-\vert u_{>N}\vert^{p+2}-\vert u_{\leq N}\vert^{p+2})\big\|_{L_{t,x}^1}. \label{i}
	\end{equation}
	
For term $II$, we use $\{f,g\}_P=\nabla \text\O(fg)+\text{\O}(f\nabla g)$. When the derivative hits the product, we integrate by parts. We find that $II$ contributes to the right-hand side of \eqref{flmor start} a multiple of
	\begin{align}
	&\big\|\tfrac{1}{\vert x\vert}u_{\leq N}[F(u)-F(u_{\leq N})]\big\|_{L_{t,x}^1} \label{ii}
	\\ &+\big\|\nabla u_{\leq N}[F(u)-F(u_{\leq N})]\big\|_{L_{t,x}^1}. \label{iii}
	\end{align}

Finally, for $III$, we integrate by parts when the derivative hits $u_{>N}.$ We find that $III$ contributes to the right-hand side of \eqref{flmor start} a multiple of
	\begin{align}
	&\big\|\tfrac{1}{\vert x\vert}u_{>N}P_{\leq N}\big(F(u)\big)\big\|_{L_{t,x}^1} \label{iv}
	\\&+\big\|u_{>N}\nabla P_{\leq N}\big(F(u)\big)\big\|_{L_{t,x}^1}. \label{v}
	\end{align}

Thus, continuing from \eqref{flmor start}, we see that to complete the proof of Theorem~\ref{uppbd} it will suffice to show that
	\begin{equation}\label{btmq}
	\norm{\text{Mor}}_{L_t^\infty(I)}\lesssim_u \eta N^{1-2s_c}
	\end{equation}
and that the error terms \eqref{i} through \eqref{v} are acceptable, in the sense that they can be controlled by $\eta(N^{1-2s_c}+K_I).$

To prove \eqref{btmq}, we use Bernstein, \eqref{hct1}, \eqref{equ:hardy} to estimate
	\begin{align*}
	\norm{\text{Mor}}_{L_t^\infty(I)}&\lesssim\xonorm{\vert\nabla\vert^{-1/2}\nabla u_{>N}}{\infty}{2}\xonorm{\vert\nabla\vert^{1/2}(\tfrac{x}{\vert x\vert}u_{>N})}{\infty}{2}
	\\ &\lesssim\xonorm{\vert\nabla\vert^{1/2}u_{>N}}{\infty}{2}^2
	\lesssim_u \eta N^{1-2s_c}.
	\end{align*}

We next turn to the estimation of the error terms \eqref{i} through \eqref{v}.
For \eqref{i}, we write
	\begin{align}
	\eqref{i}&\lesssim \xonorms{\tfrac{1}{\vert x\vert}(u_{\leq N})^{p+1}u_{>N}}{1}\label{ib1}
	\\ &\quad+\xonorms{\tfrac{1}{\vert x\vert}u_{\leq N}(u_{>N})^{p+1}}{1}.\label{ib2}
	\end{align}

For \eqref{ib1}, we use H\"older, Hardy, the chain rule, Bernstein, \eqref{assume1.1}, \eqref{hct}, and \eqref{sct} to estimate
	\begin{align*}
	\xonorms{\tfrac{1}{\vert x\vert}(u_{\leq N})^{p+1}u_{>N}}{1}
	&\lesssim\xonorm{\tfrac{1}{\vert x\vert}(u_{\leq N})^{p+1}}{2}{4/3}
		\xonorm{u_{>N}}{2}{4}
	\\ &\lesssim \xonorm{\nabla(u_{\leq N})^{p+1}}{2}{4/3}
		\xonorm{u_{>N}}{2}{4}
	\\ &\lesssim\xonorm{u}{\infty}{\frac{4p^2}{3p-2}}^p\xonorm{\nabla u_{\leq N}}{2}{2p}
	\xonorm{u_{>N}}{2}{4}\\
&\lesssim\big\|\nsc u_{\leq N}\big\|_{L_t^2L_x^4}\xonorm{u_{>N}}{2}{4}
	\\ &\lesssim_u \eta N^{-s_c}(1+N^{2s_c-1}K_I)\\
&\lesssim_u \eta N^{1-2s_c}(1+N^{2s_c-1}K_I),
	\end{align*}
which is acceptable, where we used  Proposition \ref{adecay1} and $u\in L_t^\infty L_x^{2p}$
to get
$$\xonorm{u}{\infty}{\frac{4p^2}{3p-2}}\lesssim 1,\quad\tfrac{4p^2}{3p-2}\in(p+1,2p),$$
and $N<1$.

For \eqref{ib2}, we consider two cases. If $\vert u_{\leq N}\vert\ll \vert u_{>N}\vert$, then we can absorb this term into the left-hand side of \eqref{flmor start}, provided we can show
	\begin{equation}\label{is it finite}
	\xonorms{\tfrac{1}{\vert x\vert}\vert u_{>N}\vert^{p+2}}{1}<\infty.
	\end{equation}
Otherwise, we are back in the situation of \eqref{ib1}, which we have already handled. Thus, to render \eqref{ib2} an acceptable error term it suffices to establish \eqref{is it finite}. To do this, we use Hardy, Sobolev embedding, Bernstein, and Lemma~\ref{fspssn} to estimate
	\begin{align*}
	\xonorms{\tfrac{1}{\vert x\vert} \vert u_{>N}\vert^{p+2}}{1}
	&\lesssim \xonorms{\vert x\vert^{-\frac{1}{p+2}}u_{>N}}{p+2}^{p+2}
	\lesssim\xonorms{\vert\nabla\vert^{\frac{1}{p+2}}u_{>N}}{p+2}^{p+2}
	\\ &\lesssim\xonorm{\vert\nabla\vert^{\frac{2p-1}{p+2}}u_{>N}}{p+2}{\frac{2(p+2)}{p+1}}^{p+2}
	 \lesssim N^{1-2s_c}\xonorm{\vert\nabla\vert^{s_c} u}{p+2}{\frac{2(p+2)}{p+1}}^{p+2}
	\\ &\lesssim_u N^{1-2s_c}\Big(1+\int_I N(t)^{2}\,dt\Big)<\infty.
	\end{align*}

We next turn to \eqref{ii}. Writing
	\begin{align*}
	\xonorms{\tfrac{1}{\vert x\vert}u_{\leq N}[F(u)-F(u_{\leq N})]}{1}
	\lesssim\xonorms{\tfrac{1}{\vert x\vert}(u_{\leq N})^{p+1}u_{>N}}{1}
	+\xonorms{\tfrac{1}{\vert x\vert}u_{\leq N}(u_{>N})^{p+1}}{1},
	\end{align*}
we recognize the error terms that we just estimated, namely \eqref{ib1} and \eqref{ib2}. Thus \eqref{ii} is acceptable.

For \eqref{iii}, we use H\"older,   Proposition \ref{adecay1}, \eqref{assume1.1},  \eqref{hct}, and \eqref{sct} to estimate
	\begin{align*}
	\eqref{iii}&\lesssim\xonorm{\nabla u_{\leq N}}{2}{2p}
	\xonorm{u_{>N}}{2}{4}
	\xonorm{u}{\infty}{\frac{4p^2}{3p-2}}^p\\
&\lesssim \big\|\nsc u_{\leq N}\big\|_{L_t^2L_x^4}\xonorm{u_{>N}}{2}{4}\\
&\lesssim  \eta N^{-s_c}(1+N^{2s_c-1}K_I)
	\\&\lesssim \eta N^{1-2s_c}(1+N^{2s_c-1}K_I),
	\end{align*}
which is acceptable.

Finally, for \eqref{iv} and \eqref{v}, we first use Hardy to estimate
	\begin{align*}
	\eqref{iv}&+\eqref{v}
	\\&\lesssim\xonorm{u_{>N}}{2}{4}\xonorm{\tfrac{1}{\vert x\vert} P_{\leq N}\big( F(u)\big)}{2}{4/3}+\xonorm{u_{>N}}{2}{4}\xonorm{\nabla P_{\leq N}\big(F(u)\big)}{2}{4/3}
	\\&\lesssim\xonorm{u_{>N}}{2}{4}\xonorm{\nabla P_{\leq N}\big(F(u)\big)}{2}{4/3}
	\end{align*}

Thus, by \eqref{hct}, we only need to prove
	\begin{equation}
	\nonumber
	\xonorm{\nabla P_{\leq N}\big(F(u)\big)}{2}{4/3}\lesssim_u \eta N^{1-s_c}(1+N^{2s_c-1}K_I)^{1/2}.
	\end{equation}

To this end, we use H\"older, Bernstein, the fractional chain rule, \eqref{assume1.1}, \eqref{still small1}, and \eqref{still small2} to estimate
	\begin{align*}
	\xonorm{\nabla P_{\leq N}\big(F(u)\big)}{2}{4/3}
	&\lesssim N\xonorm{F(u)-F(u_{\leq N/\eta^2})}{2}{4/3}+N^{1-s_c}\xonorm{\nsc F(u_{\leq N/\eta^2})}{2}{4/3}
	\\ &\lesssim N\xonorm{u}{\infty}{2p}^p\xonorm{u_{>N/\eta^2}}{2}{4}
	+N^{1-s_c}\xonorm{u}{\infty}{2p}^p\xonorm{\nsc u_{\leq N/\eta^2}}{2}{4}
	\\ &\lesssim_u \eta N^{1-s_c}(1+N^{2s_c-1}K_I)^{1/2}.
	\end{align*}

Therefore, we
complete the proof of Theorem~\ref{uppbd}.

\section{The quasi-soliton scenario}\label{qs section}
In this section, we rule out the existence of quasi-soliton soutions, that is, solutions as in Theorem \ref{threenew} such that $\int_0^{T_{\rm max}}N(t)^{3-2s_c}\,dt=\infty.$ The proof will rely primarily on the frequency-localized  Morawetz estimate.
\begin{theorem}[No quasi-solitons]\label{no quasi-solitons}
Let $s_c>3/2$. Then there are no radial almost periodic solutions $u:[0,T_{\rm max})\times\R^4\to\C$ to \eqref{equ1.1} with $N(t)\equiv N_k\geq 1$ on each characteristic subinterval $J_k\subset[0,T_{\rm max})$ that satisfy
            \begin{equation}\label{rfcmd12}
            \nonumber
            \|u\|_{L_{t,x}^{3p}([0,T_{\rm max}))}=\infty
            \end{equation}
and
            \begin{equation}\label{rfkjfs}
            K:=\int_0^{T_{\rm max}}N(t)^{3-2s_c}\,dt=\infty.
            \end{equation}
\end{theorem}

\begin{lemma}[Lower bound]\label{lem:mor lb2} Let $u:[0,T_{\rm max})\times\R^3\to\C$ be a radial almost periodic solution as in Theorem~\ref{threenew} with $s_c>3/2$. Let $I\subset[0,T_{\rm max}).$ Then there exists $N_0>0$ such that for any $N<N_0,$ we have
	\begin{equation}
	K_I\lesssim_u \iint_{I\times\R^4}\frac{\vert u_{>N}(t,x)\vert^{p+2}}{\vert x\vert}\,dx\,dt,
	\label{eq:mor lb2}
	\end{equation}
with $K_I:=\int_IN(t)^{3-2s_c}\,dt$.
\end{lemma}

\begin{proof}
First, by the same argument as  (7.3) in \cite{MMZ},
we deduce that there exists $C(u)$ sufficiently large and $N_0>0$ so that for $N<N_0$
        \begin{equation}\label{stepping stone}
        \inf_{t\in I}N(t)^{2s_c}\int_{\vert x\vert\leq\frac{C(u)}{N(t)}}\vert u_{>N}(t,x)\vert^{2}\,dx\gtrsim_u 1.
        \end{equation}
This together with H\"older's inequality yields
	\begin{align*}
	\iint_{I\times\R^4}\frac{\vert u_{> N}(t,x)\vert^{p+2}}{\vert x\vert}\,dx\,dt&\gtrsim_u \int_I N(t)\int_{\vert x\vert\leq\frac{C(u)}{N(t)}}\vert u_{> N}(t,x)\vert^{p+2}\,dx\,dt
	\\ &\gtrsim_u\int_I N(t)^{1+2p}\bigg(\int_{\vert x\vert\leq\frac{C(u)}{N(t)}}\vert u_{> N}(t,x)\vert^2\,dx\bigg)^{\frac{p+2}{2}}\,dt
	\\ &\gtrsim_u \int_I N(t)^{1+2p}\big(N(t)^{-2s_c}\big)^{\frac{p+2}{2}}\,dt
	\gtrsim_u K_I.
	\end{align*}
Thus, we complete the proof of Lemma \ref{lem:mor lb2}.
		\end{proof}

Finally, we turn to prove Theorem \ref{no quasi-solitons}.
Suppose $u$ were such a solution. Let $\eta>0$ and let $I\subset[0,T_{\rm max})$ be a compact time interval, which is a contiguous union of characteristic subintervals.

Combining \eqref{limemd} and \eqref{eq:mor lb2}, we find that for $N$ sufficiently large, we have
	$$K_I\lesssim_u \eta(N^{1-2s_c}+K_I).$$
Choosing $\eta$ sufficiently small, we deduce $K_I\lesssim_u N^{1-2s_c}$ uniformly in $I$. We now contradict \eqref{rfkjfs} by taking $I$ sufficiently large inside of $[0,T_{\rm max})$. This completes the proof of Theorem~\ref{no quasi-solitons}. Therefore, we conclude Theorem \ref{theorem}.

\begin{center}

\end{center}


\begin{thebibliography}{99}

\bibitem{Bou}  J. Bourgan, Global wellposedness of defocusing critical nonlinear Schr\"odinger equation in the radial case. J. Amer. Math. Soc., 12(1999), 145-171.


\bibitem{Cav} T. Cazenave, Semilinear Schr\"odinger equations. Courant Lecture Notes in Mathematics,
Vol. 10. New York: New York University Courant Institute of
Mathematical Sciences, 2003. ISBN: 0-8218-3399-5.

\bibitem{CW} M. Christ and M. Weinstein, Dispersion of small
amplitude solutions of the generalized Korteweg-de Vries equation.
J. Funct. Anal., 100(1991), 87-109.


\bibitem{CKSTT07}J. Colliander, M. Keel, G. Staffilani, H. Takaoka,
and T. Tao, Global well-posedness and scattering for the
energy-cirtical nonlinear Schr\"{o}dinger equation in
$\mathbb{R}^3$. Annals of Math., 167 (2008), 767-865.


\bibitem{Dodson3} B. Dodson, Global well-posedness and scattering for the defocusing,
$L^2$-critical, nonlinear Schr\"odinger equation when $d\geq3$. J.
Amer. Math. Soc., 25 (2012), 429-463.


\bibitem{Dodson2} B. Dodson, Global well-posedness and scattering for the defocusing,
$L^2$-critical, nonlinear Schr\"odinger equation when $d = 2$.
Preprint, arXiv:1006.1375.


\bibitem{Dodson1} B. Dodson, Global well-posedness and scattering for the
defocusing, $L^2$-critical, nonlinear Schr\"odinger equation when $d
= 1$. Amer. J. Math., 138(2016), no. 2, 531-569.

\bibitem{Dodson} B. Dodson, Global well-posedness and scattering for the mass
critical nonlinear Schr\"odinger equation with mass below the mass
of the ground state. Advances in Mathematics, 285(2015), 1589-1618.


\bibitem{Dodson4} B. Dodson, Global well-posedness and scattering
for the focusing, energy-critical nonlinear Schr\"odinger problem in
dimension $d=4$ for initial data below a ground state threshold.
arXiv: 1409.1950v1.




\bibitem{DMMZ} B. Dodson, C. Miao, J. Murphy, and J. Zheng, The defocusing quintic NLS in four space dimensions,  Annales de l'Institut Henri Poincare/Analyse non lineaire, doi:10.1016/j.anihpc.2016.05.004.

\bibitem{GV} J. Ginibre and G. Velo, Smoothing properties and retarded estimates for some dispersive evolution equations. Comm. Math. Phys., \textbf{144} (1992), 163--188.


\bibitem{Gri}  M. Grillakis, On nonlinear Schr\"odinger equations. Comm. PDE, 25(2000), 1827-1844.

\bibitem{KeT98} M. Keel and T. Tao, Endpoint Strichartz estimates. Amer. J.
Math., 120(1998), 955--980.


\bibitem{KM} C. Kenig and F. Merle, Global well-posedness, scattering,
and blow-up for the energy-critical focusing nonlinear
Schr\"{o}dinger equation in the radial case, Invent. Math., 166:3
(2006), 645-675.


\bibitem{KM2010} C. Kenig and F. Merle, Scattering for $\dot{H}^\frac12$ bounded solutions
to the cubic, defocusing NLS in $3$ dimensions. Trans. Amer. Math.
Soc., 362 (2010), 1937-1962.


\bibitem{KTV2009} R. Killip, T. Tao, and M. Visan, The cubic nonlinear
Schr\"odinger equation in two dimensions with radial data. J. Eur.
Math. Soc., 11 (2009), 1203-1258.

\bibitem{KV2010} R. Killip and M. Visan,  Energy-supercritical NLS: critical
$\dot{H}^s$-bounds imply scattering. Comm. Partial Differential
Equations 35 (2010), 945-987.




\bibitem{KV20101} R. Killip and M. Visan, The focusing energy-critical nonlinear
Schr\"odinger equation in dimensions five and higher. Amer. J.
Math., 132 (2010), 361-424.


\bibitem{KV} R. Killip and M. Visan, The defocusing
energy-supercritical nonlinear wave equation in three space
dimensions, Trans. Amer. Math. Soc., 363 (2011), 3893-3934.



\bibitem{KVnote} R. Killip and M. Visan, Nonlinear Schr\"odinger
equations at critical regularity. Lecture notes prepared for Clay
Mathematics Institute Summer School, Z$\ddot{u}$rich, Switzerland,
2008.


\bibitem{KVZ2008} R. Killip, M. Visan, and X. Zhang. The mass-critical nonlinear
Schr\"odinger equation with radial data in dimensions three and
higher. Analysis and Partial Differential Equations, 1, no. 2 (2008)
229-266.





\bibitem{LinStr} J. Lin and W. Strauss, Decay and scattering of solutions of a nonlinear Schr\"odinger equation. J. Funct. Anal., 30(1978), 245-263.


\bibitem{MMZ} C. Miao, J. Murphy, and J. Zheng, The defocusing energy-supercritical {NLS} in four space dimensions. J. Funct. Anal., 267(2014), no. 6, 1662-1724.


\bibitem{Mu} J. Murphy, Inter-critical NLS: critical $\dot{H}^s_x$-bounds imply scattering. SIAM
J. Math. Anal., 46(2014), 939-997.

\bibitem{Mur2} J. Murphy, The defocusing $\dot{H }^{1/2}$-critical NLS in high dimensions. Discrete Contin. Dyn.
Syst., Series A 34(2014), 733-748.

\bibitem{Mur} J. Murphy, The radial defocusing nonlinear Schr\"odinger equation in three space
dimension, Comm. Partial Differential Equations, 40(2015), 265-308.



\bibitem{St} R. S. Strichartz, Restriction of Fourier transform to
quadratic surfaces and decay of solutions of wave equations. Duke
Math. J. 44(1977), 705-774.

\bibitem{Tao} T. Tao, Global well-posedness and scattering for the higher-dimensional
energy-critical non-linear Schr\"odinger equation for radial data. New York J. of Math. \textbf{11} (2005), 57-80.


\bibitem{TVZ2007} T. Tao, M. Visan, and X. Zhang, Global well-posedness and
scattering for the mass-critical nonlinear Schr\"odinger equation
for radial data in high dimensions. Duke Math. J., 140 (2007)
165-202.

\bibitem{Visanphd} M. Visan, The defocusing energy-critical
nonlinear Schr\"odinger equation in dimensions five and higher. Ph.
D Thesis, UCLA, 2006.

\bibitem{Visan2007} M. Visan, The defocusing energy-critical nonlinear Schr\"odinger
equation in higher dimensions. Duke Math. J., 138 (2007) 281-374.


\bibitem{Visan2011} M. Visan, Global well-posedness and scattering for the defocusing
cubic NLS in four dimensions. Int. Math. Res. Not. 2011 (2011), doi:
10.1093/imrn/rnr051.

\end{thebibliography}
\end{document}